\documentclass[10pt,a4paper]{article}
\usepackage{fullpage}
\usepackage{amsfonts,amsmath,amssymb}
\usepackage{amsthm}
\usepackage{graphicx}\usepackage{sectsty}

\theoremstyle{plain}
\newtheorem{theorem}{Theorem}[section]
\newtheorem{lemma}[theorem]{Lemma}

\theoremstyle{definition}
\newtheorem{definition}[theorem]{Definition}
\theoremstyle{remark}

\numberwithin{equation}{section}

\sectionfont{\large}
\newenvironment{acknowledgement}[1][Acknowledgement
]{\begin{trivlist} \item[\hskip \labelsep {\bfseries
#1}]}{\end{trivlist}}

\begin{document}
\title{A Near-Field Basis in Radially Symmetric Interior Transmission Problem}
\author{Lung-Hui Chen$^1$}\maketitle\footnotetext[1]{Department of
Mathematics, National Chung Cheng University, 168 University Rd.
Min-Hsiung, Chia-Yi County 621, Taiwan. Email:
mr.lunghuichen@gmail.com;\,lhchen@math.ccu.edu.tw. Fax:
886-5-2720497.}
\begin{abstract}
The spectrum of interior transmission problem is the zero set of certain entire functional determinant. It is classic that we deploy the series of exponential polynomials to approximate the distribution of the roots of the entire functions of exponential type.  
We construct an exponential system in the form of $\{e^{ik_jr}\}$ according to the set of interior transmission eigenvalues $\{k_{j}\}$. The eigenvalues are the zeros of a sine-type function. In particular, they are intersection points of two asymptotically periodic entire functions. The intersection set is asymptotically sine-like near the real axis, so we may manage to construct a basis according to the class of spectral objects. Due to the result of Paley-Wiener theorem, the zero set generates a natural duality in the form of Fourier transform associated with exponential polynomials. Whenever there is a sufficient quantity of transmission eigenvalues, we are given a series of exponential polynomials to saturate the functional density, which completes a $L^{2}$-Riesz basis in a suitable ball. 
\\MSC: 35P25/35R30/34B24.
\\Keywords:  interior transmission eigenvalue/completeness of exponential polynomial/functional analysis.
\end{abstract}
\section{Preliminaries and Main Result}
In this paper, we study a basis theory of the following eigenvalue problem:
\begin{eqnarray}\label{1.1}
\left\{%
\begin{array}{ll}
    \Delta w+k^2n(x)w=0,  & \hbox{ in }D; \vspace{3pt}\\\vspace{3pt}
    \Delta v+k^2v=0, & \hbox{ in }D; \\\vspace{3pt}
    w=v, & \hbox{ on }\partial D; \\\vspace{3pt}
    \frac{\partial w}{\partial \nu}=\frac{\partial v}{\partial \nu},& \hbox{ on }\partial D,\\
\end{array}%
\right.
\end{eqnarray}
where $k\in\mathbb{C}$;
 $\nu$ is the unit outer normal;
   $D$ is the unit ball in $\mathbb{R}^3$;
   $n(x)=n(|x|)\in\mathcal{C}^2(\mathbb{R}^+)$;
   $n(x)>0$ for all $x\in D$;
   $n(x)=1$, for $x\notin D$. We set $|x|:=r\in\mathbb{R}^+$ and $\hat{x}:=(\theta,\varphi)\in\mathbb{S}^2$ as the spherical coordinate. Here we consider a radially symmetric perturbation.
The equation~(\ref{1.1}) is the homogeneous interior transmission eigenvalue problem. We say $k\in\mathbb{C}$ is an interior transmission eigenvalue of~(\ref{1.1}) if there is a non-trivial solution pair $w,v\in\mathcal{C}^2(D)\cap\mathcal{C}^1(\bar{D})$ to the boundary value problem.

\par
The equation in the form~(\ref{1.1}) was first introduced by Colton and Monk \cite{Colton4,Colton} and Kirsch \cite{Kirsch86}, in which some denseness properties of the far-field patterns are examined. For a comprehensive study of the problem, we refer to \cite{Aktosun,Cakoni,Cakoni2,Colton3,Colton2,Colton5,L,La,Liu,Mc,Rynne}.
We also refer certain Weyl's types of asymptotics to \cite{Chen,Chen3,Chen5,Colton5,L,La}.
The eigenvalue problem occurs when the plane waves are perturbed by the inhomogeneity specified by the index of refraction $n(x)$.  The inverse problem is to determine the index of refraction by the measurements of the scattered waves in the far-fields. The  problem plays a role in various disciplines of science and technology such as sonar and radar, geophysical sciences, medical imaging, remote sensing, and non-destructive testing in device manufacturing.
Many inverse uniqueness and existence problems can be reduced to a problem in the form of~(\ref{1.1}) \cite{Colton,Kirsch86,K,Liu}. 
The interior transmission eigenvalues also play a role in the numerical computation and theoretical analysis in the inverse
scattering theory. A Riesz basis provides a numerical scheme to analyze the convergence and the stability of approximating series such as those in the finite element method \cite{Colton8,Sun2}, optimization methods, and Fourier expansion method \cite{K}. Theoretically, we can expand the index of refraction uniquely in the series of certain exponential polynomials connecting to its own eigenvalues.
\par
 Let $\{k_j\}$
be the collection of interior transmission eigenvalues. For the numerical and functional analysis of the inverse scattering problem in acoustic and electromagnetic waves, we study the completeness and minimality of the exponential system $\{e^{ik_jr}\}$ from the point of view of \cite{Levin,Levin2}. In this paper, we will show that if there is a sufficient quantity of zeros of $$D_{0}(k):=\frac{\sin{k}}{k}y'(1;k)-\cos{k}\,y(1;k),$$
 then we have a $L^2$-basis in the form of exponential polynomials $\{e^{ik_jr}\}$ in the region that contains the perturbation, which would be the only near-field basis in the interior transmission problem known to the author.

\par
Let us expand the solution $(v,w)$ of~(\ref{1.1}) in two series of spherical harmonics,  \cite[p.\,109]{Colton4} and \cite[p.\,227]{Colton2}:
\begin{eqnarray}\label{1.3}
&&v(x;k)=\sum_{l=0}^{\infty}\sum_{m=-l}^{m=l}a_{l,m}j_l(k r)Y_l^m(\hat{x});\\
&&w(x;k)=\frac{1}{r}\sum_{l=0}^{\infty}\sum_{m=-l}^{m=l}b_{l,m}y_l(r)Y_l^m(\hat{x}), \label{1.4}
\end{eqnarray}
where $r:=|x|$; $\hat{x}=(\theta,\varphi)\in\mathbb{S}^2$;
$j_l$ is the spherical Bessel function of the first kind of order $l$. The summations converge uniformly and absolutely on the compact subsets of $|x|=r\geq1$.

\par
Thus, we deduce the following homogeneous system from the boundary condition of~(\ref{1.1}): $-l\leq m\leq l$, $l=0,1,2,\ldots$,
\begin{eqnarray}\vspace{5pt}\nonumber
\left\{
  \begin{array}{ll}
    &a_{l,m}j_l(k r)|_{r=1}-b_{l,m}\frac{y_l(r)}{r}|_{r=1}=0; \label{1.6}\vspace{10pt}\\
    &a_{l,m}[j_l(k r)]'|_{r=1}-b_{l,m}[\frac{y_l(r)}{r}]'|_{r=1}=0.
  \end{array}
\right.
\end{eqnarray}

\par
Let us define
\begin{equation}\label{D}\nonumber
D_{l}(k):=\det\left(%
\begin{array}{cc}
  j_l(k r)|_{r=1}  & -\frac{y_l(r)}{r}|_{r=1}\vspace{6pt}\\
  \{j_l(k r)\}'|_{r=1}& -\{\frac{y_l(r)}{r}\}'|_{r=1}\\
\end{array}%
\right),
\end{equation}
and the existence of each nontrivial $(a_{l,m},b_{l,m})$ of~(\ref{1.6}) is equivalent to finding the zeros of
\begin{equation}\nonumber
D_{l}(k)=0.
\end{equation}
We refer the details to \cite{Aktosun,Chen,Chen3,Colton3,Colton2,Colton5}.
For $l=0$,
\begin{equation}\label{1.7}
D_{0}(k)=\frac{\sin{k}}{k}y'(1;k)-\cos{k}\,y(1;k),
\end{equation}
in which
we consider the Liouville transformation of
$y(r)=y(r;k)$:
\begin{eqnarray}\label{}
&z(\xi):=[n(r)]^{\frac{1}{4}}y(r),\mbox{ where
}\xi:=\int_0^r[n(\rho)]^{\frac{1}{2}}d\rho.
\end{eqnarray}
In particular, we define
\begin{eqnarray}
&&B(r):=\int_0^r[n(\rho)]^{\frac{1}{2}}d\rho,\,0\leq r\leq1.\\
&&B(1):=B.\nonumber
\end{eqnarray}
Then we describe the following Sturm-Liouville theory.
\begin{eqnarray}
\left\{%
\begin{array}{ll}
    z''+[k^2-p(\xi)]z=0,\,0<\xi<B; \vspace{5pt}\\
    z(0)=0;\,z'(0)=[n(0)]^{-\frac{1}{4}}, \nonumber
\end{array}%
\right.
\end{eqnarray}
where
\begin{equation}
p(\xi)=\frac{n''(r)}{4[n(r)]^2}-\frac{5}{16}\frac{[n'(r)]^2}{[n(r)]^3}.\nonumber
\end{equation}
Hence \cite[(2.2)]{Chen5},
\begin{equation}\nonumber
p(\xi)\not\equiv0.
\end{equation}
\par
With  \cite[p.16]{Po}, we need the following asymptotics:
\begin{equation}\label{1.13}
z(\xi;k)=\frac{\sin k\xi}{k}-\frac{\cos
k\xi}{2k^2}Q(\xi)+\frac{\sin
k\xi}{4k^3}[p(\xi)+p(0)-\frac{1}{2}Q^2(\xi)]+O(\frac{\exp[|\Im
k|\xi]}{k^4}),
\end{equation}
where $Q(\xi)=\int_0^\xi p(\xi)d\xi$;
\begin{equation}\label{1.14}
z'(\xi;k)= \cos k\xi+\frac{\sin k\xi}{2k}Q(\xi)+\frac{\cos
k\xi}{4k^2}[p(\xi)-p(0)-\frac{1}{2}Q^2(\xi)]+O(\frac{\exp[|\Im
k|\xi]}{k^3}).
\end{equation}
Here, we note that the boundary condition in~(\ref{}) is
$z'(0)=[n(0)]^{-\frac{1}{4}}$, so we need a multiple
$[n(0)]^{\frac{1}{4}}$ before $z(\xi;k)$ and $z'(\xi;k)$ in the
asymptotics above. Thus,
\begin{eqnarray}\nonumber
D_0(k)&=&\frac{\sin k}{k}y'(1;k)- (\cos
k)y(1;k)\vspace{5pt}\\
&=&\frac{\sin(1-B)k}{[n(0)]^{\frac{1}{4}}k}+\frac{\sin(1-B)k}{2[n(0)]^{\frac{1}{4}}k}[O_1(k)+O_2(k)]
+\frac{\sin(1+B)k}{2[n(0)]^{\frac{1}{4}}k}[O_2(k)-O_1(k)],\label{1.15}
\end{eqnarray}
where
\begin{equation}\label{116}
O_1(k):=1-\frac{\cot
kB}{2k}Q(B)+\frac{1}{4k^2}[p(B)+p(0)-\frac{1}{2}Q^2(B)]+O(\frac{1}{k^3}),
\end{equation}
which is bounded for $\Im k\neq0$. Similarly,
\begin{equation}\label{117}
O_2(k):=1+\frac{\tan
kB}{2k}Q(B)+\frac{1}{4k^2}[p(B)-p(0)-\frac{1}{2}Q^2(B)]+O(\frac{1}{k^3}).
\end{equation}
Therefore,
\begin{eqnarray}\label{119}
ik^3D_0(k)=k^2[e^{i(1-B)k}-e^{-i(1-B)k}][1+O(\frac{1}{k})]-p(0)[e^{i(1+B)k}-e^{-i(1+B)k}][1+O(\frac{1}{k})],
\end{eqnarray}
in which $\sin kB\neq0$, and $\cos kB\neq0$ due to the poles in~(\ref{116}) and~(\ref{117}). In addition, the asymptotics holds suitably away from the real axis. Accordingly, we can apply Rouch\'{e}'s theorem in complex analysis to conclude that $D_0(k)$ has the same number of zeros and poles as $k^2[e^{i(1-B)k}-e^{-i(1-B)k}]-p(0)[e^{i(1+B)k}-e^{-i(1+B)k}]$ in a suitable strips that contain the real axis but away from the origin. In particular, the Wilder's type of theorem can be applied to study the distribution of the zeros of $k^2[e^{i(1-B)k}-e^{-i(1-B)k}]-p(0)[e^{i(1+B)k}-e^{-i(1+B)k}]$. More details can be found in \cite{Chen,Chen3,Chen5}. 

\par
To discuss the distributional rules of the zeros of $D_{0}(k)$, we introduce a definition as follows \cite{Levin,Levin2}.
\begin{definition}
Let $f(z)$ be an integral function of order $\rho$, and
$n(f,\alpha,\beta,r)$ denotes the number of the zeros of $f(z)$
inside the angle $[\alpha,\beta]$ and $|z|\leq r$. We define the
density function as
\begin{equation}\nonumber
\Delta_f(\alpha,\beta):=\lim_{r\rightarrow\infty}\frac{n(f,\alpha,\beta,r)}{r^{\rho}},
\end{equation}
and
\begin{equation}\nonumber
\Delta_f(\beta):=\Delta_f(\alpha_0,\beta),
\end{equation}
with some fixed $\alpha_0\notin E$ such that $E$ is at most a
countable set.
\end{definition}
Referring to the previous results \cite{Chen,Chen3,Chen5}, we have
\begin{theorem}\label{11}
The determinant $D_0(k)$ is an entire function of order $1$ and of
type $1+B$. In particular, the angular eigenvalue densities in $\mathbb{C}$ are equal to
\begin{eqnarray}
&&\Delta_{D_0}(-\epsilon,\epsilon)=\Delta_{D_0}(\pi-\epsilon,\pi+\epsilon)=\frac{1+B}{\pi};\label{113}\\
&&\Delta_{D_0}(\epsilon,\pi-\epsilon)=\Delta_{D_0}(\pi+\epsilon,2\pi-\epsilon)=0.\label{114}
\end{eqnarray}
\end{theorem}
The theorem implies the following result.
\begin{theorem}\label{13}
Let $\{k_j\}$ be the collection of all the interior transmission eigenvalues of~(\ref{1.1}).
The exponential system $\{e^{ik_j r}\}$ is a Riesz basis in $L^2(U)$, where $U$ is a closed ball centered at $0$ with radius $1+B$.
\end{theorem}

\section{A Proof of Theorem \ref{13}}
Now we show that $D_0(k)$ is actually a sine-type function as in \cite[p.\,163]{Levin2}, and verify the four conditions of this class of functions.
\begin{lemma}\label{221}
The indicator diagram is $[-(1+B)i,(1+B)i]$.
\end{lemma}
\begin{proof}
The proof is straight forward from~(\ref{1.7}).
\begin{eqnarray}\label{DDD}
D_{0}(k)&=&\frac{\sin k}{k}y(1;k)\{\frac{y'(1;k)}{y(1;k)}-k\frac{\cos k}{\sin k}\}\\
&=&\frac{\sin k}{k}y(1;k)\{k\frac{\cos kB[1+O(\frac{1}{k})]}{\sin k B[1+O(\frac{1}{k})]}-k\frac{\cos k}{\sin k}\},\nonumber
\end{eqnarray}
away from the real axis. The bracket term is only of $O(k)$-growth away from the real axis. In this case, we may compute the Lindel\"{o}f's indicator function \cite[p.\,70]{Levin}, and  deduce that$$h_{D_{0}}(\theta)=h_{\frac{\sin k}{k}}(\theta)+h_{y(1;k)}(\theta),\,\theta\neq0,$$
in which we applied an asymptotic identity from \cite[p.\,159]{Levin}. Moreover, we note that $h_{\frac{\sin k}{k}}(\theta)=|\sin\theta|$ and $h_{y(1;k)}(\theta)=|B||\sin \theta|$, by considering~(\ref{1.13}), and~(\ref{1.14}). Hence, we deduce that
\begin{equation}\label{DD}
h_{D_0}(\theta)=(1+B)|\sin\theta|.
\end{equation}
In the case that $\frac{y'(1;k)}{y(1;k)}-k\frac{\cos k}{\sin k}\equiv0$ in~(\ref{DDD}), we deduce that $$D_{0}(k)\equiv0,$$
which implies that $n\equiv1$ from the result of \cite[Theorem\,3.1]{Aktosun}. That is, the eigenfunction pair $(w,v)$ is trivial. We refer more details to \cite{Chen,Chen3}.
Considering~(\ref{DD}), we prove the lemma by Cartwright theory \cite[p.\,251]{Levin}. 

\end{proof}

\begin{lemma}\label{222}
All interior transmission eigenvalues are located in a suitable strip along the real axis.
\end{lemma}
\begin{proof}
We refer the proof to \cite[Theorem 4.2]{Chen5}.
\end{proof}
An alternative proof to Lemma \ref{222} is provided by Theorem 2.5 below.
\begin{lemma}\label{223}
There exist constant $h,\,c,$ and $C$ such that
$$0<c<|D_0(x+ih)|<C<\infty,\,-\infty<x<\infty.$$
\end{lemma}
\begin{proof}
In this case, we let $|h|$ be large so that $D_0(x+ih)$ has no zero on $|\Im k|= h $. $D_0(k)$ is horizontally bounded by~(\ref{1.7}),~(\ref{1.13}), and~(\ref{1.14}).

\end{proof}
In general, we do not have the second condition required in \cite[p.\,163]{Levin2}. Thus, we need the following lemma.
\begin{lemma}\label{224}
There are  simple interior transmission eigenvalues of~(\ref{1.1}) that satisfy the separation condition: $\inf_{j\neq j'}|k_j-k_{j'}|=2\delta>0$, with density $\frac{2(1+B)}{\pi}$ in $\mathbb{C}$.
\end{lemma}
\begin{proof}
The zero density of $D_0(k)$ is $\frac{2(1+B)}{\pi}$ by~(\ref{113}) and~(\ref{114}). Now we divide $D_0(k)$ by finitely many zeros, and deduce  again an entire function of order $1$ and of type $1+B$ as follows: The Wilder's type of theorem \cite{Chen3,Chen5,Dickson,Dickson2} combined with~(\ref{119}) implies the following result.
\begin{theorem}
Let $R_{11}$, $R_{12}$, and $R_{13}$ be three adjacent strips containing all but a finite exception of the zeros of $D_0(k)$ along the positive real axis, and
$$N(R_{1i}(\beta,s,K)), \,i=1,2,3,$$  be the zero counting function inside the strip starting with $\Im k=\beta$, of length $s$, and of width $K$.
There exist some $K>0$ and large $\beta$ such that
\begin{eqnarray}
&&N(R_{11}(\beta,s,K))\sim \left\{%
\begin{array}{ll}
    s\frac{B}{\pi}, & 1\geq B; \vspace{5pt}\\
    s\frac{1}{\pi}, &1<B; \label{2.2} \\
\end{array}%
\right.    \vspace{5pt}\\
&&N(R_{12}(\beta,s,K))\sim s\frac{|1-B|}{\pi};\vspace{5pt}\label{2.3} \\
&&N(R_{13}(\beta,s,K))\sim \left\{%
\begin{array}{ll}
    s\frac{B}{\pi}, & 1\geq B; \vspace{5pt}\\
    s\frac{1}{\pi}, &1<B. \\
\end{array}%
\right.\label{2.4}
\end{eqnarray}
\end{theorem}
The estimate~(\ref{119}) implies~(\ref{2.2}),~(\ref{2.3}), and~(\ref{2.4}). We refer the detailed proof to \cite[Sec.\,4]{Chen3}. For our application here, we take $\beta$ large and suitable $s$ for each~(\ref{2.2}),~(\ref{2.3}), and~(\ref{2.4}) such that there is exactly one zero in each square. Hence, we have an entire function with all simple zeros disregarding the finitely many ones near the origin. This proves the lemma.

\end{proof}
\begin{proof}
Now Lemma \ref{221}, Lemma~\ref{222}, Lemma~\ref{223}, and Lemma~\ref{224} imply that $D_0(k)$ is a sine-typed entire function of exponential type with the indicator diagram $[-(1+B)i,(1+B)i]$. Let $\{k_j\}$ be the collection of all interior transmission eigenvalues of~(\ref{1.1}). Then the system $\{e^{ik_jr}\}$ is a Riesz basis in $L^2(-(1+B),(1+B))$ by \cite[p.\,170]{Levin2}, and thus in $L^2(U)$,  due to the radial symmetry assumption.

\end{proof}
\begin{acknowledgement}
The author wants to thank Prof. Chao-Mei Tu at NTNU for proofreading an earlier version of this manuscript.
\end{acknowledgement}

\end{document}